\documentclass{article}
\usepackage[english]{babel}
\usepackage{amsfonts, amsmath, amsthm, amssymb,amscd,indentfirst}
\usepackage{amsmath,amssymb,latexsym,indentfirst}
\usepackage{times}
\usepackage{palatino}

\newtheorem{theorem}{Theorem}[section]

\newtheorem{proposition}{Proposition}[section]

\newtheorem{definition}{Definition}[section]

\newtheorem{example}{Example}[section]
\newtheorem{remark}{Remark}[section]

\def\Cc{{\cal C}}

\def\Cc{{\cal C}}

\def\e{{\rm e}}

\def\b{\beta}

\def\oveh{{\overline\eh}}
\def\eh{{{\bf{\rm{e}}}}}

\def\Kc{{\cal K}}
\def\a{{\alpha}}

\def\g{{\gamma}}

\def\Nc{{\cal N}}

\def\Wc{{\cal W}}

\def\Hc{\mathfrak{r}}

\def\Cc{{\cal C}}

\def\eh{{{\bf{\rm{e}}}}}

\def\a{{\alpha}}

\def\g{{\gamma}}

\def\Wc{{\cal W}}

\def\Hc{\mathfrak{r}}

\def\Cc{\mathcal{C}}

\def\eh{{{\bf{\rm{e}}}}}

\def\a{{\alpha}}

\def\g{{\gamma}}

\def\Hc{{\cal H}}

\begin{document}

\title{Positive mass and Penrose type inequalities for asymptotically hyperbolic hypersurfaces}

\author{Levi Lopes de Lima\thanks{Federal University of Cear\'a,
Department of Mathematics, Campus do Pici, R. Humberto Monte, s/n, 60455-760,
Fortaleza/CE, Brazil. Partially supported by CNPq/Brazil and FUNCAP/CE.}
\and Frederico Gir\~ao\thanks{Federal University of Cear\'a,
Department of Mathematics, Campus do Pici, R. Humberto Monte, s/n, 60455-760,
Fortaleza/CE, Brazil.}}

\maketitle

\begin{abstract}
We establish versions of the Positive Mass and Penrose inequalities  for a class of asymptotically hyperbolic hypersurfaces. In particular, under the usual dominant energy condition, we prove in all dimensions $n\geq 3$ an optimal Penrose inequality for certain graphs in hyperbolic space $\mathbb H^{n+1}$ whose boundary has constant mean curvature  $n-1$.
\end{abstract}

\section{Introduction}\label{intro}

As  a result of investigations on its Hamiltonian formulation, General Relativity has provided Riemannian Geometry with a notion of mass, denoted by $\mathfrak m_{(X,g)}$, which is an  invariant defined in terms of  the asymptotic behavior of a noncompact Riemanniann manifold $(M,g)$ arising as  a (time-symmetric) initial data set. Roug\-hly speaking, it is  assumed that $(M,g)$ converges at infinity to some model geometry $(N,g_0)$ and the invariant is engineered so as to somehow  measure the corresponding rate of convergence. In particular, the important question arises as to whether, under a suitable dominant energy condition, the invariant in question satisfies the inequality
\begin{equation}\label{pmineq}
\mathfrak m_{(M,g)}\geq 0,
\end{equation}
with equality taking place if and only if $(M,g)=(N,g_0)$ isometrically.

The classical example is the asymptotically flat case, where the model geometry at infinity is Euclidean. Here, $\mathfrak m_{(M,g)}$ is the so-called {\em ADM mass} and the famous Positive Mass Conjecture (PMC) says that $\mathfrak m_{(M,g)}\geq 0$  if one assumes that $R_g$, the scalar curvature of $(M,g)$, is non-negative, with the equality taking place if and only if $(M,g)$ is isometric to Euclidean space. This has been proved by Schoen-Yau \cite{SY} if $n\leq 7$ and by Witten \cite{W} in the spin case. Furthermore, if $(M,g)$ carries a compact inner boundary $\Gamma$, the so-called Penrose Conjecture (PC) improves the PMC by stating that
\begin{equation}\label{penrose}
\mathfrak m_{(M,g)}\geq \frac{1}{2}
\left(
\frac{A}{\omega_{n-1}}
\right)^{\frac{n-2}{n-1}},
\end{equation}
where $A$ is the area of $\Gamma$ and $\omega_{n-1}$ is the area of the unit $(n-1)$-sphere, with equality holding if and only if $(M,g)$ is the Schwarzschild solution. Here it is  assumed that $\Gamma$ is a (possibly disconnected) outermost minimal hypersurface that corresponds to the horizon of a collection of black holes inside $M$. If $n=3$ the PC has been verified for $\Gamma$ connected by Huisken-Ilmanen \cite{HI} and in general by Bray \cite{Br}. More recently, Bray and Lee \cite{BL} established the conjecture for $n\leq 7$ with the extra requirement that $M$ is spin for the rigidity statement.  Even though many partial results have been obtained \cite{BI} \cite{H1} \cite{Sc} \cite{FS} \cite{J}, the conjecture remains wide open in higher dimensions except for a recent breakthrough for Euclidean graphs by Lam \cite{L1} \cite{L2}. Inspired by his technique, the authors \cite{dLG} were
able to establish Penrose type inequalities for the ADM mass of a large class of asymptotically flat hypersurfaces in certain Riemannian manifolds with a warped product structure at \lq spatial\rq\, infinity.
In particular, Lam's result was extended to Euclidean  quasi-graphs.
One of the purposes of the present note is precisely to indicate how the methods introduced in \cite{dLG} can be adapted to the setting of asymptotically hyperbolic hypersurfaces.

In recent years, motivated by a renewed interest in negative cosmological constant solutions of Einstein field equations  in connection with the celebrated AdS/CFT correspondence, there has been much work toward  defining similar invariants for complete non-compact Riemannian manifolds whose geometry at infinity asymptotes some model geometry  other than the Euclidean one. A notable example occurs in case
the model is (locally) hyperbolic; see for instance \cite{CH}, \cite{CN}, \cite{H2} and \cite{M}. Here, the situation is a bit more complicated because in general the naturally defined invariant is not a number but instead a linear functional on a certain finite dimensional space of functions on the model. In some cases, however, it is possible to extract a mass-like invariant (i.e. a real number) out of the functional, so it makes sense  to ask whether inequalities similar to (\ref{pmineq}) and (\ref{penrose}) hold, with the corresponding rigidity statement.

Starting with the seminal work of Min-Oh \cite{Mi}, much effort has been made toward understanding the case in which the geometry at infinity is (locally) hyperbolic, with various positive mass  inequalities and rigidity results being proved under natural geometric assumptions; see for instance \cite{AD}, \cite{Wa}, \cite{CH} and \cite{ACG}. On the other hand, progress in the presence of an inner boundary $\Gamma$ is apparently much harder to obtain and the only results known to the authors are the contribution in \cite{CH} to the effect that in the spin case the mass is strictly positive (with no explicit bound) if $\Gamma$ has mean curvature at most $n-1$ and the recent preprint by Dahl-Gicquaud-Sakovich \cite{DGS}, where by using the ideas first presented in \cite{L1}, the authors establish suboptimal Penrose type inequalities  for certain hyperbolic graphs in the case $\Gamma$ is minimal. In this paper we adapt the method introduced in \cite{dLG} to establish positive mass and Penrose type inequalities for a large class of asymptotically hyperbolic hypersurfaces (Definition \ref{asympdef}). We now briefly describe the results and postpone a detailed presentation to  Subsection \ref{descr}. Recall that the main ingredient in \cite{dLG} is a flux-type formula that goes back to Reilly \cite{R} and has been developed along the years by several authors \cite{ARS} \cite{Ro} \cite{ABC} \cite{dL} \cite{AdLM}. The identity says roughly that the (extrinsic) scalar curvature of a hypersurface in an Einstein manifold endowed with a Killing field is, up to a multiplicative angle factor, the divergence of the vector field given by its Newton tensor applied to the tangential component of the Killing field. Assuming that the hypersurface is asymptotically flat in a suitable sense, integration of the identity over larger and larger domains reveals that the total flow of the vector field over the sphere at infinity equals the ADM mass of the hypersurface, and since this is also given by a bulk integral involving the scalar curvature, we were able to draw many interesting positive mass and Penrose like inequalities. As explained in Sections \ref{geo} and \ref{penroptimal}, the same principle works in the asymptotically hyperbolic case, which first gives a general mass formula (Theorem \ref{another2}) and then a positive mass inequality (Theorem \ref{another3}). Moreover, in the case of \lq balanced\rq\, graphs in hyperbolic space $\mathbb H^{n+1}$ carrying a horizon with constant mean curvature $n-1$, an extra argument as in \cite{dLG} is carried out to give an {\em optimal} Penrose inequality in {\em all} dimensions (Theorem \ref{penroptimal0}). We stress that no such optimal Penrose inequality seems to be available in the literature. In particular, up to the corresponding rigidity statement, this settles, for this class of manifolds,  a well-know conjecture.

We remark that the optimal Penrose inequality follows from a general formula for the mass of asymptotically hyperbolic hypersurfaces (not necessarily graphs); see Remark \ref{last}. Also, explicit mass formulae are also obtained for other kinds of hypersurfaces with  constant mean curvature horizons; see Remarks \ref{really} e \ref{finallyy}.

\vspace{0.3cm}

\noindent
{{\bf Acknowledgements:} The authors would like to thank Fernando Marques for many enlightening conversations during the preparation of this paper.}

\section{Preliminaries and statement of the results}\label{preli}

In this section we collect  the basic facts on  mass-like invariants of asymptotically hyperbolic manifolds and state our mains results (Theorems \ref{another2}, \ref{another3} and \ref{penroptimal0} below). Standard references for the material covered here are \cite{CN}, \cite{CH}, \cite{H2} and \cite{M}.

\subsection{Mass-like invariants for asymptotically hyperbolic manifolds: a review}\label{review}

We start by describing the model geometries at infinity. Thus, let us fix once and for all a closed $(n-1)$-dimensional manifold $(E,h)$. We assume that the scalar curvature $R_h$ of $(E,h)$ is constant, which we normalize so that  $R_h=(n-1)(n-2)\epsilon$, with $\epsilon=0,\pm 1$. On the product manifold $F=E\times (r_0,+\infty)$, $r_0>1$, we consider the metric
\begin{equation}\label{metric}
b_{\epsilon}=\frac{dr^2}{\rho(r)^2}+r^2h, \quad \rho(r)=\sqrt{r^2+\epsilon},
\end{equation}
where $r$ is the standard linear coordinate in $(r_0,+\infty)$.
We remark that $(F,b_{\epsilon})$ has constant scalar curvature, namely, $R_{b_{\epsilon}}=-n(n-1)$. Moreover, $b_{\epsilon}$ is Einstein if and only if $h$ is.
We also fix a (local) orthonormal frame $\{\mathfrak e_{a}\}_{a=1}^{n-1}$ in $(E,h)$, so that $\{\tilde\eh_{\alpha}\}_{\alpha=1}^n$ given by
\begin{equation}\label{frame}
\tilde\eh_a=r^{-1}\mathfrak e_a,\quad \tilde\eh_n=\rho\frac{\partial}{\partial r},
\end{equation}
is a (local) orthonormal frame in $F$.

Roughly speaking, a Riemannian manifold is asymptotically hyperbolic if its metrics approaches the model $(F,b_\epsilon)$  in a suitable manner as one goes to infinity. The following definition formalizes this idea.

\begin{definition}\label{be}\cite{CN}\cite{CH}
A complete $n$-dimensional manifold $(M,g)$ is said to be {\em asymptotically hyperbolic (AH)} if there exists a compact set $K\subset M$ and a diffeomorphism $\Psi:M-K\to  F$ such that
\begin{equation}\label{chart}
\sum_{\alpha\beta}|(\Psi_*g)_{\alpha\beta}-\delta_{\alpha\beta}|
+\sum_{\alpha\beta\gamma}|\tilde\eh_\gamma((\Psi_*g)_{\alpha\beta})|=O(r^{-\tau}),
\end{equation}
for some $\tau>n/2$. Here, $(\Psi_*g)_{\alpha\beta}$ are the coefficients of the pushed forward metric $\Psi_*g$ with respect to the frame (\ref{frame}).
\end{definition}

Note that the definition is {\em chart-dependend} in principle, so that some further work is required to justify it.
Thus assume that one has {\em two} charts $\Psi_1$ and $\Psi_2$, both satisfying (\ref{chart}).  It is clear then that
$
\Psi_{12}=\Psi_2\circ\Psi_1^{-1}
$
satisfies
\begin{equation}\label{transf3}
\Psi_{12}^*b_{\epsilon}=b_{\epsilon}+O(r^{-\tau}).
\end{equation}
The following result shows that the hyperbolic structure at infinity in Definition \ref{be} is well defined as it does not depend on the chart used to express it.

\begin{proposition}\label{equiv}\cite{CN}\cite{CH}
If $\Phi:F\to F$ is a diffeomorphism satisfying
\begin{equation}\label{transf4}
\Phi^*b_{\epsilon}=b_{\epsilon}+O(r^{-\tau})
\end{equation}
then there
exists an isometry $A$ of $(F,b_\epsilon)$, possibly defined only for $r$ large, so that
\begin{equation}\label{isom}
\Phi=A+O(r^{-\tau}),
\end{equation}
with a corresponding assertion for the first and second order derivatives.
\end{proposition}

Applying this to $\Phi=\Psi_{12}$ and using (\ref{transf3}), we see the
that being AH is  a well-defined notion indeed.
With these preliminaries out of the way, we now pass to the definition of mass-like invariants for this class of manifolds. For simplicity of notation we set $b=b_{\epsilon}$ and consider the vector space
$$
\Nc_{b}=\{\varphi\in C^{\infty}(F); {\rm Hess}_{b}\varphi=\varphi({\rm Ric}_{b}+nb)\}.
$$
In case $\Nc_{b}$ is non-trivial, the {\em mass functional} of an AH manifold $(M,g)$ is defined, with respect to a {\em given} chart $\Psi$, as being the linear function $\mathfrak m_{\Psi}:\Nc_{b}\to\mathbb R$,
\begin{equation}\label{mass}
\mathfrak m_{\Psi}(\varphi)=\lim_{r\to +\infty}c_n\int_{E_r}\left(\varphi\left({\rm div}_be-d{\rm tr}_be\right)-
i_{\nabla^b\varphi}e+({\rm tr}_be)d\varphi\right)(\nu_{r})dE_{r},
\end{equation}
where $e=\Psi_*g-b$, $\nu_r$ is the unit normal to $E_r=E\times \{r\}$ pointing toward infinity and
$$
c_n=\frac{1}{2(n-1)\omega_{n-1}}.
$$
Standard arguments show that the limit in (\ref{mass}) exists and is finite if $\Psi$ is {\it admissible} in the sense
that (\ref{chart}) is satisfied with $\tau>n/2$ and the difference
\begin{equation}\label{diff}
\mathfrak R_g=R_{\Psi_*g}-R_b
\end{equation}
between scalar curvatures is integrable.
We remark for further reference that, granted this, the mass functional $\mathfrak m_{\Psi}$ can indeed be computed with respect to {\em any} orthonormal frame $\{\eh_\alpha\}$ along $(F,b)$ by means of the following recipe:
\begin{equation}\label{recipe}
\mathfrak m_{\Psi}(\varphi)=\lim_{r\to +\infty}c_n\int_{E_r}J(\varphi)_\alpha\nu_\alpha dE_r,
\end{equation}
where
\begin{equation}\label{recipe2}
J(\varphi)_\alpha=\varphi\left(e_{\alpha\beta,\beta}-e_{\beta\beta,\alpha}\right)
-e_{\alpha\beta}\varphi_\beta+e_{\beta\beta}\varphi_\alpha,
\end{equation}
with $\varphi_\alpha=\eh_\alpha(\varphi)$ and $e_{\alpha\beta,\gamma}=\eh_\gamma(e_{\alpha\beta})$.

However, the question remains of relating this chart-dependent definition  for two distinct  admissible  charts at infinity, say $\Psi_1$ and $\Psi_2$. We start by observing that, by Proposition \ref{equiv}, (\ref{transf3}) implies that
\begin{equation}\label{transffin}
\Psi_{12}=A+O(r^{-\tau}),
\end{equation}
for some isometry $A$ of $(F,b)$.

\begin{proposition}\label{solve}\cite{CH}\cite{CN}
If $\Psi_1$ and $\Psi_2$ are admissible charts at infinity such that (\ref{transffin}) holds, then
\begin{equation}\label{amb}
\mathfrak m_{\Psi_2}(\varphi)=\mathfrak m_{\Psi_1}(\varphi\circ A^{-1}),
\end{equation}
for any $\varphi\in \Nc_b$.
\end{proposition}

This result makes it clear the difficulty of extracting geometric information out of the family of functional $\mathfrak m_{\Psi}$, with $\Psi$ running over the set of admissible charts, since the indicated action of the isometry group of $b$ on $\Nc_b$ shows up as one passes from one chart to another. Thus, a detailed knowledge of the structure of the action is required in order to proceed.
In this regard  we now discuss two important examples.

\begin{example}\label{another}{\rm
It is shown in \cite{CN} that if either $\epsilon=-1$ and ${\rm Ric}_h<0$ or $\epsilon=0$ and $(E,h)$ is a flat space form then $\dim \Nc_b=1$, with $\Nc_b$ being generated by $\rho$.  In particular, $\mathfrak m_{\Psi}(\rho)$ does {\em not} depend on the chosen chart $\Psi$ and this common value is, by definition, the {\em mass} of $(M,g)$, denoted $\mathfrak m_{(M,g)}$.
No positive mass inequality seems to be known here and in Theorem \ref{another3} below we provide such an inequality in the setting of AH hypersurfaces, assuming that the corresponding dominant energy condition holds.}
\end{example}

\begin{example}\label{subtle}{\rm
A much subtler case takes place when $(E,h)$ is the unit $(n-1)$-sphere with the standard round metric, so that $(F,b_1)$, $F=E\times(0,+\infty)$, is hyperbolic space $\mathbb H^{n}$. In this case, $\Nc=\Nc_{b_1}$ is generated by  $\{\rho^{(i)}\}_{i=0}^n$, where $\rho^{(i)}=z_i$ are the linear coordinates in Lorentz space $\mathbb L^{n+1}$ seen as functions on the standard hyperboloid model $\mathbb H^{n}\subset \mathbb L^{n+1}$.
We note that $\rho=\rho^{(0)}$.
Here, the background isometry group $O^+(n,1)$ acts naturally on $\Nc$ preserving the metric
\begin{equation}\label{lorent}
( z,w)=z_0w_0-z_1w_1-\cdots -z_nw_n,
\end{equation}
with $\{\rho^{(i)}\}$ as an orthonormal basis.
We provide $\Nc$ with a time orientation by declaring that $\rho^{(0)}$ is future direct.
Thus, if we set, for any admissible $\Psi$,
\begin{equation}\label{admi}
P_{i}=\mathfrak m_{\Psi}(\rho^{(i)}),\quad i=0,1,\cdots, n,
\end{equation}
then, as explained in \cite{CH}, the only {\em chart-independent} information available out of $P$ are its causal character, past/future pointing nature and the numerical invariant
\begin{equation}\label{masshyp}
\mathfrak m^2_{(M,g)}=\left|P_{0}^2-\sum_{\a=1}^n P_{\a}^2\right|.
\end{equation}
This was also considered by Wang \cite{Wa} in case $(X,g)$ is conformally compact.
We note that if
$$
\Cc^+=\{w\in \Nc;(w,w)=0,w_0>0\}
$$ is the future-directed light-cone and
$
\Nc^+
$
is its interior, then
$P$ defined in (\ref{admi}) is timelike future-directed if and only if $\mathfrak m_{\Psi}(\varphi)>0$ for any $\varphi\in\Nc^+$.
In low dimensions and in the spin case it has been proved  under the usual dominant energy assumption $R_g\geq -n(n-1)$
that $P$ is either time-like future directed or zero, with the latter case holding if and only if $(M,g)$ is isometric to hyperbolic space; see \cite{ACG} \cite{CH} \cite{Wa}. Thus, whenever $P=\mathfrak m_{\Psi}$ is time-like, it is natural to choose the sign of $\mathfrak m_{(M,g)}$ so as to coincide with that of $P_{0}$. With this choice the above mentioned rigidity result says that $\mathfrak m_{(M,g)}>0$ unless that $(M,g)$ is hyperbolic space (where the mass vanishes).
We also remark that it is proven in \cite{CH} that $\mathfrak m_{(M,g)}>0$ if additionally $(M,g)$ carries a black hole horizon $\Gamma$ whose mean curvature  is at most $n-1$.
In any case, if we {\em assume} that $P$ is timelike future directed then, as already observed in \cite{DGS}, the mass can be rewritten as
\begin{equation}\label{faju}
\mathfrak m_{(M,g)}=\inf_{\varphi\in \Nc^1}\mathfrak m_{\Psi}(\varphi),
\end{equation}
where
$$
\Nc^1=\{\varphi\in\Nc^+;(\varphi,\varphi)=1\}
$$
is the unit hyperpoloid in $\Nc^+$.
Equivalently, we can always replace $\Psi$ with $A\circ\Psi$, where $A$ is a hyperbolic isometry, and assume that
\begin{equation}\label{faju2}
\mathfrak m_{(M,g)}=\mathfrak m_{\Psi}(\rho).
\end{equation}
Charts with this property are called {\em balanced}. Starting from (\ref{faju2}) we will be able to establish an optimal Penrose inequality for certain AH quasi-graphs graphs in hyperbolic space; see Theorem \ref{penroptimal0} and Remark \ref{last}.
}
\end{example}

\subsection{AH hypersurfaces: a description of the results}\label{descr}

In this subsection we describe our general setup and state the main results in the paper.

Let $F=E\times(r_0,+\infty)$ be an $n$-dimensional model space endowed with the reference metric (\ref{metric}) as in the previous subsection and
consider the warped product $(\overline F,\overline b)$, with $\overline F=F\times {{I}}$ and
\begin{equation}\label{metric2}
\overline b=b+\rho^2dt^2,
\end{equation}
where $t$ is the standard linear coordinate in $I\subset \mathbb R$. It is well-known that the assumption $\rho\in \Nc_b$ is equivalent to  $(\overline F,\overline b)$ being  Einstein. Notice that each $t\in I$ defines a horizontal slice $F_{t}=F\times \{t\}\hookrightarrow \overline F$ which is totally geodesic, so that $F_t=F$ isometrically. This follows easily from the fact that $X=\partial/\partial t$, the vertical coordinate field, is Killing. Notice moreover that from $\rho=|X|_{\overline b}$ we find that
\begin{equation}\label{framevert}
\eh_{0}=\rho^{-1}X
\end{equation}
is the unit normal to the slices. We finally consider an $(n+1)$-dimensional Riemannian manifold $(\overline M,\overline g)$ endowed with a globally defined Killing field $\overline X$. We assume that there exists a closed subset $\overline C\subset \overline M$ such that $\overline M-\overline C$ is {\em isometric} to our warped product model $(\overline F,\overline b)$, with   $\overline X$ corresponding to $X$ under the identification given by the isometry.

\begin{definition}\label{asympdef}
Let $(\overline M,\overline g)$ be as  above.
A complete, isometrically immersed hypersurface $(M,g)\looparrowright (\overline M,\overline g)$, possibly with an inner boundary $\Gamma$, is {\em asymptotically hyperbolic (AH)} if  there exists a compact subset $K\subset M$ such that $F_M=M-K$, the end of $M$, can be written as a vertical graph over some slice $F\hookrightarrow \overline M-\overline C$, with the graph being  associated to a smooth function $u:F\to\mathbb R$ such that
\begin{equation}\label{chart2}
\sum_\alpha\left|\rho u_\alpha\right|+\sum_{\alpha\beta}
\left|\rho_\beta u_\alpha+\rho u_{\alpha\beta}\right|=O(r^{-\frac{\tau}{2}}),
\end{equation}
for some $\tau>n/2$,
where $u_\alpha=\tilde\eh_\alpha(u)$, $u_{\alpha\beta}=\tilde\eh_{\beta}(\tilde\eh_\alpha (u))$, etc.
Moreover, we assume that $\mathfrak R_g=R_{{\Psi_u}_*g}-R_b$ is integrable, where $\Psi^{-1}_u(x)=(x,u(x))$, $x\in F$.
\end{definition}

The decay conditions (\ref{chart2}) are tailored so that,  by the remarks in Subsection \ref{review} and (\ref{mett1}) below, it makes sense to compute, for $\varphi\in \Nc_{b_{\epsilon}}$, the mass $\mathfrak m_{\Psi_u}(\varphi)$, where $\Psi_u$ is the graph coordinate chart in the definition. But notice that, in general, this number has no invariant meaning due to the transformation rule (\ref{amb}).

We further assume that $M\looparrowright\overline M$ is {\em two-sided} in the sense that it carries a globally defined  unit normal $N$, which we choose so that $N=\eh_{n+1}$ at infinity.
This allows us to consider the {\em angle function} $\Theta_{\overline X}=\langle \overline X,N\rangle:M\to\mathbb R$ associated to $\overline X$.
We then say that $\Theta_{\overline X}$ {\em does not change sign} if $\Theta_{\overline X}\geq 0$ long $M$.

The following theorem computes the mass $\mathfrak m_{\Psi_u}(\rho)$, where $\Psi_u$ is the graph representation at infinity of an AH hypersurface.

\begin{theorem}\label{another2}
If $(M,g)\looparrowright(\overline M,\overline g)$ is as in Definition \ref{asympdef} and $\Gamma=\emptyset$ then
\begin{equation}\label{form1}
\mathfrak m_{\Psi_ u}(\rho)=c_n\int_M \left(2S_2\Theta_{\overline X}+{\rm Ric}_{\overline g}(\overline X^T,N)\right) dM.
\end{equation}
\end{theorem}

Here, $S_2$ is the $2$-mean curvature of $M$; see (\ref{twomean}) below. The following positive mass inequality is then an immediate consequence.

\begin{theorem}\label{another3}
If, in addition to the hypothesis of Theorem \ref{another2}, we are under the conditions of Example \ref{another} and, moreover, $(\overline M,\overline g)$ is Einstein, i.e ${\rm Ric}_{\overline g}=-n\overline g$,  then
\begin{equation}\label{an}
\mathfrak m_{(M,g)}=c_n\int_M \Theta_{\overline X}\mathfrak R_g dM,
\end{equation}
where $\mathfrak R_g=R_g+n(n-1)$.
In particular, if $\Theta_{\overline X}$ does not change sign and $\mathfrak R_g\geq 0$ outside  the zero set of $\Theta_{\overline X}$ then $\mathfrak m_{(M,g)}\geq 0$.
\end{theorem}

Theorem \ref{another3} follows from Theorem \ref{another2} and (\ref{eins}) below with $\lambda=-n$, after noticing that $\mathfrak m_{(M,g)}=\mathfrak m_{\Psi_u}(\rho)$ by  Example \ref{another}.

We
now discuss Penrose-like inequalities in the context of Example \ref{subtle}.
In the presence of an outermost minimal horizon $\Gamma\subset M$, the conjectured inequality reads as
\begin{equation}\label{penrhypger}
\mathfrak m_{(M,g)}\geq \frac{1}{2}\left[\left(\frac{A}{\omega_{n-1}}\right)^{\frac{n-2}{n-1}}+
\left(\frac{A}{\omega_{n-1}}\right)^{\frac{n}{n-1}}
\right];
\end{equation}
see \cite{BC} and \cite{Ma} for details and also for a discussion of the corresponding rigidity results. Recently, versions of (\ref{penrhypger}) have been proved in \cite{DGS}  for certain AH graphs in hyperbolic space $\mathbb H^{n+1}$ under the usual dominant energy condition. Here we will be mainly interested in the case $\Gamma$ has constant mean curvature equal to $n-1$, where the conjectured inequality assumes the classical form (\ref{penrose}); see \cite{Wa}, \cite{BC} and \cite{Ma}. This is proved here for a class of graphs in $\mathbb H^{n+1}$.
To describe the result, we consider the metric (\ref{metric}) in $\mathbb H^{n+1}=\mathbb H^n\times \mathbb R$
and
an AH graph $M$ given by a function $u:\mathbb H^n\to\mathbb R$ as in Definition \ref{asympdef}. Following \cite{DGS} we say that $M$ is {\em balanced} if
$\Psi_u$ is balanced in the sense of Example \ref{subtle}. For $d\in\mathbb R$ we also consider the horosphere $\Hc_{d,\pm}$ given  as the graph of the function
\begin{equation}\label{balhor}
v(x)=d\pm\log \rho(x), \quad x\in\mathbb H^n.
\end{equation}
Any horosphere in this family is said to be {\em  balanced}.

With this notation at hand, we now state the optimal Penrose inequality for balanced AH graphs.

\begin{theorem}\label{penroptimal0}
Let $(M, g)\subset   (\mathbb H^{n+1}, \overline b)$ be a balanced AH graph as above and assume that $M$ carries an inner boundary
$\Gamma$ lying in some balanced horosphere $\Hc$. Assume further that $M$ meets $\Hc$ orthogonally along $\Gamma$ and that $\Gamma\subset \Hc$ is mean convex with respect to its inward unit normal. Then, if $R_g\geq -n(n-1)$,
\begin{equation}\label{penrchic}
\mathfrak m_{(M,g)}\geq \frac{1}{2}\left(\frac{A}{\omega_{n-1}}\right)^{\frac{n-2}{n-1}}.
\end{equation}
\end{theorem}

\begin{remark}\label{ortho}
{\rm
The orthogonality assumption easily implies that the mean curvature of $\Gamma\subset M$ is $n-1$, so that $\Gamma$ is a horizon indeed.
}
\end{remark}

\begin{remark}\label{expli}
{\rm It will be convenient to
consider  the Poincar\'e disk model for $\mathbb H^n$, so that
$$
\mathbb H^n=\{x\in \mathbb R^n; |x|<1\},
$$
$$
b=\frac{4}{(1-|x|^2)^2}(dx_1^2+\cdots +dx_n^2)
$$
and
$\Nc$ is generated by
$$
\rho(x)=\frac{1+|x|^2}{1-|x|^2}, \quad \rho^{(\alpha)}=\frac{2x_\alpha}{1-|x|^2},\quad \alpha=1,\cdots,n.
$$
Notice that
\begin{equation}\label{relation}
\rho^2-\sum_{\alpha=1}^n(\rho^{(\alpha)})^2=1.
\end{equation}
Moreover, we can isometrically embed $\mathbb H^n$ into the standard half-space model
$$
\mathbb H^{n+1}_u=\left\{y=(y_1,\cdots,y_{n+1})\in\mathbb R^{n+1};y_{n+1}>0\right\}
$$
as the unit upper hemisphere centered at the origin. This embedding extends to an isometry between our original model $(\mathbb H^{n+1},\overline b)$ and $\mathbb H^{n+1}_u$ explicitly given by
$$
\Upsilon(x,s)=e^s\left(\frac{2x}{1+|x|^2},\frac{1-|x|^2}{1+|x^2|}
\right),\qquad s\in\mathbb R.
$$
Thus we see that in $\mathbb H^{n+1}_u$ the Killing field corresponding to $\overline X=\rho\eh_0$
is the radial vector field and the
horospheres in the family $\Hc_{d,+}$ (respectively, $\Hc_{d,-}$) are horizontal hyperplanes (respectively, spheres tangent to the hyperplane $y_{n+1}=0$ at the origin).
}
\end{remark}

\section{The geometry of graphs in warped products}\label{geo}

As a preparation for the proof of Theorem \ref{another2}, we now consider
a two-sided asymptotically flat  hypersurface
$(M,g)\looparrowright(\overline M,\overline g)$
as in Definition \ref{asympdef}. If $\overline \nabla$ is the Riemannian connection of $(\overline M,\overline g)$, let us denote by  $B=-\overline \nabla N$ the shape operator of $M$ with respect to its unit normal vector $N$ and by  $k_1,\ldots,k_n$  the eigenvalues of $B$ with respect to $g$ (the principal curvatures). Define
\begin{equation}\label{onemean}
S_1=\sum_ik_i
\end{equation}
and
\begin{equation}\label{twomean}
S_2=\sum_{i<j}k_ik_j.
\end{equation}
These are respectively the {\em  mean curvature} and the $2$-{\em mean curvature} of $M$. Notice that from Gauss equation we have
\begin{equation}\label{ger}
R_g=R_{\overline g}-2{\rm Ric}_{\overline g}(N,N)+2S_2.
\end{equation}
In particular, if $\overline M$ is Einstein, ${\rm Ric}_{\overline g}=\lambda \overline g$, this reduces to
\begin{equation}\label{eins}
R_g=(n-1)\lambda+2S_2.
\end{equation}
Also, we define  the {\em Newton tensor} by
\begin{equation}\label{zero}
G=S_1I-B,
\end{equation}
where $I$ is the identity map.

Later on we will need the expressions of some of these invariants along the end  $F_M$ of $M$ which, by Definition \ref{asympdef}, is a graph over the slice  $F \hookrightarrow\overline F\subset \overline M$. In the following calculations we agree on the index ranges $\alpha,\beta,...=1,\cdots,n$, $i,j,...=0,1,\cdots, n$ and use the summation convention over repeated indexes. We start by noticing that,
given a local orthonormal frame $\{\eh_\a\}_{\a=1}^n$ in $F$, we  may extend it in the usual manner to a (local) orthonormal frame $\{\eh_{i}\}_{i=0}^n$ in $F$
by adding
(\ref{framevert}).
The following proposition describes the structure equations associated to such a frame.

\begin{proposition}\label{struct}
If $\overline\nabla$ is the Riemannian connection of $\overline F$  then
\begin{equation}\label{structw}
\overline\nabla_{\eh_\a}\eh_0=0,\quad \overline\nabla_{\eh_0}\eh_\a=\rho^{-1}\rho_\a\eh_0,\quad
\overline\nabla_{\eh_0}\eh_0=-\rho^{-1}\nabla^b \rho,
\end{equation}
where $\nabla^b$ is the gradient operator of $(E,b)$.
\end{proposition}

\begin{proof}
As remarked above, the slices $F_t$ are totally geodesic and this immediately gives the first equation
in (\ref{structw}).
From this we get
\begin{eqnarray*}
\overline\nabla_{\e_0}\e_\a=\overline\nabla_{\e_\a}\e_0+[\e_0,\e_\a] & = & [\rho^{-1}\partial_t,\e_\a]\\
& = &
\rho^{-1}[\partial_t,\e_\a]-\e_\a(\rho^{-1})\partial_t\\
& = & -\e_\a(\rho^{-1})\partial_t,
\end{eqnarray*}
and the second equation follows. Finally,
$$
\overline\nabla_{\e_0}\e_0=\rho^{-2}\overline\nabla_{\partial/\partial t}
\frac{\partial}{\partial_t},
$$
and since $\partial/\partial t$ is Killing, this implies
$$
\langle \overline\nabla_{\e_0}\e_0,\e_0\rangle=\rho^{-3}
\left\langle\overline\nabla_{\partial/\partial t}\frac{\partial}{\partial t},
\frac{\partial t}{\partial t}\right\rangle=0,
$$
so that
$$
\overline\nabla_{\e_0}\e_0=\gamma^\a\e_\a,
$$
with
$$
\gamma^\a=\langle\overline\nabla_{\e_0}\e_0,\e_\a\rangle=-\rho^{-2}
\langle\overline\nabla_{\e_a}\partial_t,\partial_t\rangle=-\frac{\rho^{-2}}{2}
\e_\a\left(\rho^2\right)=-\rho^{-1}\eh_\a(\rho),
$$
as desired.
\end{proof}

Let us now write
$$
E_M=\left\{(x,u(x));x\in F\right\}\subset\overline M,
$$
as the graph associated to a smooth function $u:F\to\mathbb R$ as in Definition \ref{asympdef}.
In terms of the frame in Proposition \ref{struct}, $TE_M$ is spanned by
\begin{equation}\label{frame2}
Z_\a=u_\a\frac{\partial}{\partial t}+\e_\a=\rho u_\a\e_0+\e_\a,\quad \a=1,\cdots,n,
\end{equation}
and we  choose
\begin{equation}\label{normal}
N=\frac{1}{W}\left(\e_{0}-\rho \nabla^b u\right),
\end{equation}
where
\begin{equation}\label{asymw}
W=\sqrt{1+\rho^2|\nabla^b u|^2_b}=1+O(|x|^{-\tau}),
\end{equation}
as the unit normal to $E_M$. Notice that this is consistent with our global choice of unit normal to $M$, which is dictated by the requirement that $N=\eh_{0}$ at infinity.
Also, the induced metric on $E_M$ is
\begin{equation}\label{mett1}
g_{\a\beta}=\delta_{\a\beta}+\rho^2u_\a u_\beta,
\end{equation}
and
its inverse is
\begin{equation}\label{mett2}
g^{\a\beta}=\delta_{\a\beta}-\frac{\rho^2}{W^2}u_\a u_\beta.
\end{equation}

\begin{proposition}\label{data}
The shape operator $B$ of the graph $E_M$  with respect to the frame (\ref{frame2}) is given by
\begin{eqnarray}\label{shape}
W B_{\a\g} & = & \rho u_{\a\g}+\rho_\a u_\g+\rho_\g u_\a+\rho^2u_\a u_\g\langle\nabla^b\rho,\nabla^bu\rangle
                 -\nonumber\\
&  & \, - \frac{\rho^2}{W^2}u_\a u_\b\left(\rho u_{\b\g}+\rho_\b u_\g+\rho_\g u_\b+\rho^2u_\b u_\g\langle\nabla^b\rho,\nabla^bu\rangle\right).
\label{data2}
\end{eqnarray}
\end{proposition}

\begin{proof}
We start by computing the coefficients
$$
S_{\b\g}=\langle \overline\nabla_{Z_\b}Z_\g,N\rangle
$$
of the second fundamental form $S$ of $E_M$. Since $\overline\nabla_{\e_b}\e_0=0$ a direct computation gives
$$
\overline\nabla_{Z_\b}Z_\g=\rho u_\b\e_0(\rho u_\g)\e_0+\rho^2u_\b u_\g\overline\nabla_{\e_0}\e_0+
\rho u_\b\overline\nabla_{\e_0}\e_\g+
\e_\b(\rho u_\g)\e_0+\overline\nabla_{\e_\b}\e_\g.
$$
But notice that $\e_0(\rho u_\g)=\rho^{-1}\partial_t(\rho u_\g)=0$. Moreover, the fact that the slices are totally geodesic implies
$
\overline\nabla_{\e_\b}\e_\g=\nabla^b_{\e_\b}\e_\g,
$
and since we may assume that $\nabla^b_{\e_j}\e_k=0$ at the point where we are doing the computation, it follows that $\overline\nabla_{\e_\b}\e_\g=0$. Thus,
$$
\overline\nabla_{Z_\b}Z_\g=\rho^2u_\b u_\g\overline\nabla_{\e_0}\e_0+
\rho u_\b\overline\nabla_{\e_0}\e_\g+
\e_\b(\rho u_\g)\e_0,
$$
so that Proposition \ref{struct} and (\ref{normal})
easily give
$$
S_{\b\g}=\frac{1}{W}\left(\rho u_{\b\g}+\rho_\b u_\g+\rho_\g u_\b+\rho^2u_\b u_\g\langle \nabla^b \rho,\nabla^b u\rangle\right).
$$
The expression (\ref{data2}) for the shape operator $B_{a\g}=g^{\a\b}S_{\b\g}$  follows readily.
\end{proof}

The following proposition is a key ingredient in our approach to the mass of AH hypersurfaces, as  it shows that the specific combination of extrinsic data yielding the Newton tensor of a graph simplifies considerably after evaluation on the tangential
component of the vertical Killing field.

\begin{proposition}\label{key2}
If $F_M$ is as above then the coefficients of $G{X}^T$ with respect to the frame (\ref{frame2}) are given by
\begin{equation}\label{endup2}
(GX^T)_\a=\frac{\rho^3}{W^3}\left(u_{\b\b}u_\a-u_{\a\b}u_\b\right)
+\frac{\rho^2}{W^3}
\left(\rho_\b u_\a u_\b-\rho_\a u_\b u_\b\right).
\end{equation}
In particular, $GX^T=O(r^{-\tau+1})$.
\end{proposition}

\begin{proof}
We have
\begin{equation}\label{tt}
(GX^T)_\a=B_{\b\b}X^T_\a-B_{\a\b}X^T_\b,
\end{equation}
where
\begin{equation}\label{exp2}
X^T=X_{a}^TZ_\a=X^T_\a \oveh_\a+X_\a^T\rho f_\a\eh_0.
\end{equation}
We rewrite (\ref{data2}) as
$$
B_{\a\b} = \sum_{s=1}^8B_{\a\b}^{(s)},
$$
where
$WB_{\a\b}^{(1)}=\rho u_{\a\b}$, $WB_{\a\b}^{(2)}=\rho_\a u_\gamma$, etc.
To proceed further notice that, since $\langle X,N\rangle=\rho/W$,
$$
X^T=X-\frac{\rho}{W}N=\frac{\rho^3 |\nabla^b u|^2}{W^2}\e_0+\frac{\rho^2}{W^2}u_\a\eh_\a,
$$
and comparing with (\ref{exp2}) we obtain
\begin{equation}\label{new2}
X^T_\a=\frac{\rho^2}{W^2}{f_\a}.
\end{equation}
This yields a remarkable cancelation in (\ref{tt}) since  $B_{\b\b}^{(s)}X^T_\a=B^{(s)}_{\a\b}X^T_\b$ for $s\geq 4$. Finally, the last assertion follows from (\ref{chart2}), (\ref{asymw}) and the fact that $\rho=O(r)$ at infinity.
\end{proof}

\section{The proofs of Theorems \ref{another2}, \ref{another3} and \ref{penroptimal0}}\label{penroptimal}

In this section we prove our main results presented in Subsection \ref{descr}. As remarked in the Introduction, the starting point is the flux-type formula
\begin{equation}\label{basic}
{\rm div}_gG\overline X^T=2S_2\Theta_{\overline X} + {\rm Ric}_{\overline g}(N,\overline X^T),
\end{equation}
where $(M,g)\looparrowright (\overline M,\overline g)$ is a two-sided asymptotically flat hypersurface as in Definition \ref{asympdef}, $G$ is its Newton tensor and $\overline X^T$ is the tangential component of the Killing field $\overline X$ that agrees  with $X=\partial/\partial t$ on $\overline M-\overline C$.
In this generality, (\ref{basic}) has been first obtained in \cite{ABC} in the Lorentzian setting. The Riemannian version can be found in \cite{AdLM}.

We start with Theorem \ref{another2}. For $r_0<r< +\infty$ we consider $F^{(r)}=E\times (r,+\infty)$ so that
$E_r=\partial F^{(r)}$. As usual we denote by $\nu_r$ the unit normal to $E_r$ pointing toward infinity. If $M_r=M-u(F^r)$ we obtain, after integrating (\ref{basic}) over $M$ and using the divergence theorem,
\begin{eqnarray*}
\int_M \left(2S_2\Theta_{\overline X} + {\rm Ric}_{\overline g}(N,\overline X^T)\right) dM & = & \lim_{r\to +\infty}\int_{M_r}{\rm div}_gG\overline X^T dM \\
& = & \lim_{r\to +\infty}\int_{\partial M_r}\langle GX^T,\vartheta_r\rangle\, d\partial M_r \\
& = &  \lim_{r\to +\infty}\int_{E_r} \langle GX^T,\nu\rangle dE_r,
\end{eqnarray*}
where we have used that at infinity we may replace $\vartheta d\partial M_r$
by $\nu dE_r$.

By (\ref{mett1}), (\ref{chart2}) and Proposition \ref{key2} we have
\begin{eqnarray*}
\langle GX^T,\nu\rangle¨& = & g_{\alpha\mu}(GX^T)_\a\nu_\mu \\
  & = & (GX^T)_\a\nu_\a +\rho^2 u_\a u_\mu (GX^T)_\a\nu_\mu \\
  & = & (GX^T)_\a\nu_\a + O(r^{-2\tau+1}),
\end{eqnarray*}
and, given that the $(n-1)$-area of $E_r$ is $O(r^{n-1})$, we get
\begin{eqnarray*}
\int_M \left(2S_2\Theta_{\overline X} + {\rm Ric}_{\overline g}(N,\overline X^T)\right) dM & = & \int_{E_r}(GX^T)_\a\nu_\a dE_r +\lim_{r\to +\infty} O(r^{-2\tau+n}) \\
& = & \int_{E_r}(GX^T)_\a\nu_\a dE_r,
\end{eqnarray*}
since $\tau>n/2$. Now, a straightforward computation using (\ref{recipe2}) with
$$
e_{\a\b}=g_{\a\b}-b_{\a\b}=g_{\a\b}-\delta_{\a\b}=\rho^2u_\a u_\b
$$
gives $J(\rho)_\a=W^{-3}(GX^T)_\a$ and, in view of (\ref{asymw}), this concludes the proof of Theorem \ref{another2}.

As already remarked, Theorem \ref{another3} is an immediate consequence of Theorem \ref{another2}, so we pass to the proof of Theorem \ref{penroptimal0}. Thus, as in Example \ref{subtle}, we consider
the metric (\ref{metric}) in $\mathbb H^{n+1}=\mathbb H^n\times \mathbb R$. Here, it is convenient to
choose the Poincar\'e disk model for $\mathbb H^n$ as in Remark \ref{expli}.
Now, Remark \ref{ortho} and a well-known positivity result (Theorem 4.7 in \cite{CH}) imply that $P$ in (\ref{admi}) is timelike future direct. Thus, we can use  (\ref{faju2}) and recalling that ${\rm Ric}_{\overline b}=-(n-1)\overline b$, the computation leading to (\ref{form1}) now gives an extra boundary term, namely,
$$
\mathfrak m_{(M,g)}=c_n\int_M\Theta_{\overline X}\mathfrak R_gdM-
c_n\int_\Gamma\langle G\overline X^T,\eta\rangle d\Gamma,
$$
where $\mathfrak R_g=R_g+n(n-1)$, $\overline X= \rho \eh_0$ and $\eta$ is the outward unit co-normal along $\Gamma$. Since $M$ is a graph and $\mathfrak R_g\geq 0$, we get
\begin{equation}\label{quasi}
\mathfrak m_{(M,g)}\geq
-c_n\int_\Gamma\langle G\overline X^T,\eta\rangle d\Gamma,
\end{equation}
and we are left with the task of handling the integral. To this effect we use our orthogonality assumption to expand, in terms of a local orthonormal basis
$\{\tilde\eh_l \}_{l=1}^{n-1}$ of $T\Gamma$,
$$
\overline X^T  = \langle \overline X,\eta\rangle\eta+\sum_l\langle \overline X,\tilde \eh_l\rangle\tilde \eh_l,
$$
so that
\begin{eqnarray*}
\langle G\overline X^T,\eta\rangle &  = &  \langle \overline X,\eta\rangle\langle G\eta,\eta\rangle+\sum_l\langle \overline X,\tilde \eh_l\rangle\langle G\tilde \eh_l,\eta\rangle\\
  & = & \langle \overline X,\eta\rangle(S_1-\langle B\eta,\eta\rangle)+\sum_l\langle \overline X,\tilde \eh_l\rangle\langle G\tilde \eh_l,\eta\rangle.
\end{eqnarray*}
But
$$
\langle G\tilde\eh_l,\eta \rangle=-\langle B\tilde\eh_l,\eta\rangle=\langle \overline\nabla_{\tilde \eh_l}N,\eta\rangle= -\langle N,\overline\nabla_{\tilde \eh_l}\eta\rangle,
$$
and this vanishes due to the assumption that, along $\Gamma$, the unit normal $\xi$ to
the totally umbilic horosphere $\Hc$
equals $\pm\eta$. Moreover, this computation shows that $\langle B\eta,\tilde\eh_l\rangle=0$, which implies that $\eta$ is a principal direction of $B$ with $\langle B\eta,\eta\rangle$  being the corresponding principal curvature. Thus, $S_1-\langle B\eta,\eta\rangle=s_1(\Gamma)$, the mean curvature of $\Gamma\subset \Hc$ with respect to $N$.

To proceed further we first note that, from (\ref{balhor}) and (\ref{normal}), the unit normal to the horosphere $\Hc_{d,\pm}$ is
$$
\xi_{d,\pm}=\frac{1}{\Wc}\left(\eh_0\mp\nabla^b\rho\right),\qquad \Wc=
\sqrt{1+|\nabla^b\rho|_b^2},
$$
so that $\langle \eh_0,\xi\rangle=1/\Wc$.
Hence, if we first consider the case $\Hc=\Hc_{d,+}$ and $\eta=\xi=\xi_{d,+}$, so that
$N$ points {\em outward} $\Gamma$,  we end up with
\begin{equation}\label{really0}
-\int_\Gamma\langle G\overline X^T,\eta\rangle d\Gamma=
\int_\Gamma\frac{\rho}{\Wc}  S_1(\Gamma) d\Gamma,
\end{equation}
where $S_1(\Gamma)=-s_1(\Gamma)$ is the mean curvature of $\Gamma\subset \Hc$ with respect to its inward unit normal, namely, $-N$.
On the other hand, if we choose
$$
\eh_{\alpha}=\frac{1-|x|^2}{2}\frac{\partial}{\partial x_\alpha}
$$
as our orthonormal frame in $\mathbb H^n$, a direct computation gives $\eh_\alpha(\rho)=\rho^{(\alpha)}$, so that (\ref{relation}) can be rewritten as
\begin{equation}\label{relation2}
|\nabla^b\rho|^2=\rho^2-1.
\end{equation}
Thus, $\Wc=\rho$ and we conclude that
\begin{equation}\label{final}
\mathfrak m_{(M,g)}\geq c_n\int_\Gamma S_1(\Gamma)d\Gamma.
\end{equation}
The same argument also leads to (\ref{final}) in the remaining cases. For example, if $\Hc=\Hc_{d,+}$ and $\eta=-\xi_{d,+}$ then $N$ now points {\em inward} $\Gamma$ and $s_1(\Gamma)=S_1(\Gamma)$. Thus, in any case we can use that
the intrinsic geometry of a horosphere is Euclidean and apply the well-known Alexandrov-Fenchel inequality in order to estimate from below the boundary integral in the usual manner. This completes the proof of Theorem \ref{penroptimal0}.

\begin{remark}\label{last}
{\rm
The argument above actually gives a general mass formula for a balanced asymptotically flat hypersurface $M\subset \mathbb H^{n+1}$ with an inner boundary $\Gamma$ lying in a balanced horosphere $\Hc$ and with the property that $M$ meets $\Hc$ orthogonally along $\Gamma$, namely,
\begin{equation}\label{last2}
\mathfrak m_{(M,g)}=c_n\int_M\Theta_{\overline X}\mathfrak R_g dM+
c_n\int_\Gamma S_1(\Gamma)d\Gamma.
\end{equation}
In particular, Theorem \ref{penroptimal0} holds more generally if $M$ is assumed to be a {\em quasi-graph} in the sense that $\Theta_{\overline X}$ does not change sign.
}
\end{remark}

\begin{remark}\label{really}
{\rm In the spirit of the previous remark, we can also consider the case in which $\Gamma$ lies in a hypersurface $\Kc$ which is a graph associated to a {\em constant} function. Such a hypersurface is totally geodesic in $\mathbb H^{n+1}$ (a copy of $\mathbb H^n$) and  we deduce that $\Gamma\subset M$ is minimal under the orthogonality assumption; this is of course the case treated in \cite{DGS}. If
$\xi $ is the unit normal to $\Kc$, using again (\ref{normal}) we see that $\langle\eh_0,\xi\rangle=1$ and this gives
\begin{equation}\label{new22}
\mathfrak m_{(M,g)}=c_n\int_M\Theta_{\overline X}\mathfrak R_g dM+
c_n\int_\Gamma \rho S_1(\Gamma)d\Gamma.
\end{equation}
Thus, if $\Gamma\subset \Kc$ is mean convex, $M$ is a { quasi-graph} and $\mathfrak R_g\geq 0$ outside of the zero of $\Theta_{\overline X}$ we obtain
\begin{equation}\label{new22}
\mathfrak m_{(M,g)}\geq
c_n\int_\Gamma \rho S_1(\Gamma)d\Gamma.
\end{equation}
This estimate has been obtained in \cite{DGS} for graphs and there it is their  starting point in establishing  an array of Penrose type inequalities. Thus, we see that the results in \cite{DGS} hold under this slightly more general situation.
}
\end{remark}

\begin{remark}\label{finallyy}
{\rm To illustrate the flexibility of our method we consider the case in which $\Gamma$ lies in another class $\Sigma_c$, $c\in \mathbb R$, of \lq balanced\rq\, hyperfurces. In the $\mathbb H^n\times \mathbb R$ model, these are given as graphs associated to
$$
w_c(x)=\frac{c}{\rho(x)}, \qquad x\in\mathbb H^n.
$$
Geometrically, they are the equidistant hypersurfaces to the totally geodesic hypersurface $\Sigma_0=\mathbb H^n$. Proceeding as above, and assuming whenever needed that the mass vector $P$ in (\ref{admi}) is timelime future direct, we will eventually get
$$
\mathfrak m_{(M,g)}=c_n\int_M\Theta_{\overline X}\mathfrak R_g dM+
c_n\int_\Gamma \psi_c(\rho)S_1(\Gamma)d\Gamma,
$$
where
$$
\psi_c(\rho)=\frac{\rho}{\sqrt{1+ c^2(1-\rho^{-2})}}.
$$
Since the intrinsic geometry of $\Sigma_c$ is hyperbolic,  this can be explored just as in \cite{DGS} (via Hoffman-Spruck, Minkowski, etc.) to yield Penrose type inequalities for this kind of horizon.
}
\end{remark}

\end{document}